\documentclass[12pt,reqno]{amsart}

\usepackage{amsbsy}
\usepackage{amscd}
\usepackage{amsfonts}
\usepackage{amsmath}
\usepackage{amssymb}
\usepackage{amsthm}
\usepackage{amsxtra}

\usepackage[utf8]{inputenc}
\usepackage[margin=1in]{geometry}

\usepackage{tikz-cd}
\usepackage{tikz}
\usepackage{todonotes}
\usepackage{verbatim}

\theoremstyle{definition}

\newtheorem{thm}{Theorem}[section]

\newtheorem*{conj*}{Conjecture}

\newtheorem*{cor*}{Corollary}

\newtheorem*{defn*}{Definition}

\newtheorem*{exa*}{Example}
\newtheorem{exc}{Exercise}[section]
\newtheorem*{exc*}{Exercise}
\newtheorem{lem}[thm]{Lemma}
\newtheorem*{lem*}{Lemma}

\newtheorem*{prop*}{Proposition}
\newtheorem{ques}[thm]{Question}
\newtheorem*{ques*}{Question}

\newtheorem*{rmk*}{Remark}

\renewcommand{\P}{\mathbb{P}}
\newcommand{\Q}{\mathbb{Q}}
\newcommand{\R}{\mathbb{R}}
\newcommand{\Z}{\mathbb{Z}}

\newcommand{\eps}{\varepsilon}

\newcommand{\lr}[1]{\langle #1 \rangle}

\newcommand{\mc}{\mathcal}

\newcommand{\ol}{\overline}

\newcommand{\pr}{\prime}

\newcommand{\sm}{\setminus}

\DeclareMathOperator{\GL}{GL}

\DeclareMathOperator{\PSL}{PSL}
\DeclareMathOperator{\SL}{SL}
\DeclareMathOperator{\Tr}{Tr}

\newcommand{\be}{\begin{equation*}}
\newcommand{\ee}{\end{equation*}}
\newcommand{\bex}{\begin{exc}}
\newcommand{\eex}{\end{exc}}
\newcommand{\bpf}{\begin{proof}}
\newcommand{\epf}{\end{proof}}

\title{Hecke Triangle Groups and Special Hyperbolic Elements}
\author{Karl Winsor}
\begin{document}

\begin{abstract}
We study the action of the Hecke triangle groups $G_q$ on $\lambda_q \Q(\lambda_q^2) \cup \{\infty\}$ with $\lambda_q = 2 \cos (\pi / q)$. When $q = 18$, we show the existence of infinitely many distinct orbits of fixed points of special hyperbolic elements of $G_q$. We also find new orbits for several other values of $q$. These results provide new examples of special affine pseudo-Anosov homeomorphisms on the unfoldings of regular $q$-gons. In particular, on the unfolding of the regular $18$-gon, there are infinitely many distinct Veech group orbits of directions invariant under a special affine pseudo-Anosov.
\end{abstract}

\maketitle

\section{Introduction}

The {\em Hecke triangle groups} $G_q$ are an infinite family of lattices in $\SL(2,\R)$ parametrized by an integer $q \geq 3$. The group $G_q$ is generated by the two matrices
\begin{equation} \label{eq:ST}
S = \left(\begin{matrix}
0 & -1 \\
1 & 0
\end{matrix}\right), \quad
T_q = \left(\begin{matrix} 
1 & \lambda_q \\
0 & 1
\end{matrix}\right) ,
\end{equation}
where $\lambda_q = 2\cos(\pi / q)$. Each parabolic element of $G_q$ fixes a unique {\em cusp} in $\P^1(\R) = \R \cup \{\infty\}$, and each hyperbolic element fixes a pair of distinct points in $\P^1(\R)$. Basic longstanding open questions of interest about $G_q$ include characterizing the set of cusps \cite{Ros:fraction} and the set of hyperbolic fixed points as subsets of $\P^1(\R)$.

The difficulty of characterizing the cusps of $G_q$ seems to depend mainly on the degree of the {\em invariant trace field} of $G_q$, defined by
\be
K_q = \Q(\Tr(A^2) : A \in G_q) = \begin{cases} \Q(\lambda_q), \; & q \; \text{ odd,} \\
 \Q(\lambda_q^2), \; & q \; \text{ even,} \end{cases}
\ee
which is invariant under commensurability, see for instance \cite{McM:survey}). The cusps of $G_q$ form a single orbit $G_q \cdot \infty$. By an inductive argument using the generators in (\ref{eq:ST}), we have
\begin{equation} \label{eq:cusps}
G_q \cdot \infty \subset \lambda_q \Q(\lambda_q^2) \cup \{\infty\} .
\end{equation}
When $q$ is odd, $\lambda_q \Q(\lambda_q^2) \cup \{\infty\}$ is equal to $\P^1(K_q)$. If $K_q = \Q$, that is, $q \in \{3, 4, 6\}$, then $G_q$ is commensurable with $\SL(2,\Z)$ and it is easy to show that the containment in (\ref{eq:cusps}) is an equality. When $K_q$ is quadratic, equivalently $q \in \{5, 8, 10, 12\}$, there are several proofs showing that equality holds in these cases as well, using descent arguments or a study of the SAF-invariant \cite{Leu:cusps}, \cite{McM:infinite}, \cite{McM:heights}, \cite{Pan:fraction}.

When $K_q$ has degree at least $3$, the question of characterizing the cusps of $G_q$ is wide open, and it is known that the cusps are strictly contained in $\lambda_q \Q(\lambda_q^2) \cup \{\infty\}$, see \cite{AS:special} and references therein. An important new phenomenon in this case is the existence of {\em special} hyperbolic elements of $G_q$, which have eigenvalues in $K_q$. In contrast, a typical hyperbolic element of $G_q$ has eigenvalues in a quadratic extension of $K_q$. Fixed points of special hyperbolic elements of $G_q$ lie in $\lambda_q \Q(\lambda_q^2)$ and thus provide new examples of $G_q$-orbits contained in $\lambda_q \Q(\lambda_q^2)$. Only a few examples of such orbits are known \cite{AS:special}, \cite{ABHH:special}, \cite{Bou:central}, \cite{HMTY:orbits}, and in all known examples $K_q$ has degree $3$ or $4$. One can also obtain lower bounds on the number of orbits of $G_q$ in $\lambda_q \Q(\lambda_q^2)$ with algebraic methods, by studying the reduction of $G_q$ modulo $2$ or the class number of $K_q$ \cite{AS:special}, \cite{BR:mod2}, \cite{HMTY:orbits}, \cite{Wol:mod2}.

However, in all cases where $K_q$ has degree at least $3$, it was previously unknown whether the number of orbits of $G_q$ in $\lambda_q \Q(\lambda_q^2)$ was finite. Surprisingly, for at least one value of $q$, special hyperbolic elements are abundant enough to produce infinitely many orbits.

\begin{thm} \label{thm:G18infinite}
For $q = 18$, there are infinitely many distinct $G_q$-orbits of fixed points of special hyperbolic elements of $G_q$ contained in $\lambda_q \Q(\lambda_q^2)$.
\end{thm}

Hyperbolic elements in $G_q$ with a common power have the same fixed points, and conjugating a hyperbolic element in $G_q$ only moves its fixed points within their respective $G_q$-orbits. Thus, Theorem \ref{thm:G18infinite} tells us there are infinitely many non-conjugate maximal cyclic subgroups of special hyperbolic elements in $G_q$ when $q = 18$.

Let $\mc{O}_{K_q}$ be the ring of integers in $K_q$, and let $\mc{O}_{K_q}^\ast$ be its unit group. We searched for new examples of special hyperbolic elements of $G_q$ by computing the $\lambda_q$-continued fraction expansions of many elements of $\lambda_q \mc{O}_{K_q}$ and $\lambda_q \mc{O}_{K_q}^\ast$. The examples we found (excluding $q = 18$) are shown in Table \ref{tab:Gqperiod} in terms of periodic $\lambda_q$-continued fractions, and Table \ref{tab:G18period} contains some of the examples we found for $q = 18$. In \cite{HMTY:orbits}, it was conjectured that there are exactly $2$ orbits for $G_7$ in $\P^1(K_q)$, represented by $\infty$ and $\lambda_7^2 - 1$, and distinguished by residue classes modulo $2$. A counterexample to this conjecture is provided in \cite{Bou:central}. A similar conjecture for the unit group $\mc{O}_{K_7}^\ast$ appears in \cite{RT:units}. We expect that the structure of the $G_q$-orbits in $\lambda_q \Q(\lambda_q^2) \cup \{\infty\}$ is much more complicated, even for $q = 7$. For example, all but $27$ (rational) integers in the interval $[1,10^6]$ are cusps of $G_7$. The first few exceptional integers are
\begin{center}
    671, 26197, 98335, 121380, 221444, 249976, 255730, 298572, 327023, 327068, 339794, ...
\end{center}
%\begin{align*}
%& 671, 26197, 98335, 121380, 221444, 249976, 255730, 298572, 327023, 327068, && \\ & 339794, 358623, 382764, 389145, 401478, 463097, 491513, 524306, 592444, && \\ & 
%615575, 636530, 703056, 792674, 870515, 910467, 966497, 985080 ,
%\end{align*}
and all $27$ are special hyperbolic fixed points that do not lie in the orbit of $\lambda_7^2 - 1$. Regarding $\mc{O}_{K_7}^\ast$, letting $\lambda_7^\pr = -\lambda_7^2 + 2$, the unit $\lambda_7^7 (-\lambda_7^\pr)^{-23}$ is a special hyperbolic fixed point that does not lie in the orbit of $\lambda_7^2 - 1$. Our counterexample of $671$ is also referenced in the survey \cite{McM:survey}.

Part of our motivation for studying the Hecke triangle groups $G_q$ comes from the dynamics of billiards in the regular $q$-gon and straight-line flows on the translation surfaces $(X_q,\omega_q)$ obtained by unfolding these tables. We refer to \cite{AS:special} for more details about the following discussion. The surface $(X_q,\omega_q)$ is obtained from one or two copies of a regular $q$-gon by gluing pairs of opposite parallel sides by translations. The derivatives of orientation-preserving affine automorphisms of $(X_q,\omega_q)$ form the {\em Veech group} $V_q$, which in this case is a lattice in $\SL(2,\R)$ conjugate to $G_q$ or an index $2$ subgroup of $G_q$. As a consequence, these surfaces satisfy the {\em Veech dichotomy} \cite{Vee:billiards}: every straight-line flow is either periodic or uniquely ergodic. However, the Veech dichotomy does not tell us which directions are periodic. The set of periodic directions for $(X_q,\omega_q)$ is precisely the set of cusps of $V_q$, which is in turn equal to the set of cusps of $G_q$ up to acting by an element of $\SL(2,\R)$.

The group $\GL^+(2,\R)$ acts on the moduli space of all translation surfaces of a given genus. By applying an appropriate element $M \in \GL^+(2,\R)$ to $(X_q,\omega_q)$, one can arrange that the periodic directions for $M(X_q,\omega_q)$ lie in $\P^1(K_q)$ \cite{AS:special}, \cite{CS:algebraic}. With this normalization, the directions in $\P^1(K_q)$ are precisely the directions for which the straight-line flow has vanishing SAF-invariant. Roughly speaking, the SAF-invariant measures the algebraic obstruction to being periodic. Periodic straight-line flows have zero SAF-invariant, but it is well-known that the converse is not true. A striking source of uniquely ergodic flows with zero SAF-invariant are flows invariant under a {\em special} affine pseudo-Anosov homeomorphism. The first known examples are the Arnoux-Yoccoz pseudo-Anosovs \cite{AY:pseudo}, and other families of examples are constructed in \cite{CS:special}, \cite{DS:special}. These pseudo-Anosovs are affine automorphisms whose derivative is a special hyperbolic element of the associated Veech group.

In the language of translation surfaces, Theorem \ref{thm:G18infinite} tells us that on the unfolding of the regular $18$-gon, special affine pseudo-Anosovs are abundant. Moreover, this conclusion only depends on the commensurability class of the Veech group.

\begin{thm} \label{thm:18goninfinite}
For any translation surface $(X,\omega)$ whose Veech group $V$ is commensurable to the $G_{18}$ triangle group, there are infinitely many distinct $V$-orbits of directions that are invariant under a special affine pseudo-Anosov homeomorphism of $(X,\omega)$.
\end{thm}

In particular, for a suitable $M \in \GL^+(2,\R)$, the Veech group of $M(X_{18},\omega_{18})$ acts on $\P^1(K_{18})$ and there are infinitely many distinct orbits in $\P^1(K_{18})$ for this action. To the best of our knowledge, Theorem \ref{thm:18goninfinite} provides the first known examples of lattice Veech groups that act on the projective line over their invariant trace field with infinitely many orbits. For a lattice Veech group, the invariant trace field is equal to its {\em trace field}, the field generated over $\Q$ by the traces \cite{Hoo:trace}. Note that by \cite{McM:infinite}, for any lattice Veech group with a rational or quadratic trace field $K$, after normalizing every element of $\P^1(K)$ is a cusp and there are only finitely many orbits of cusps. We also remark that the special pseudo-Anosovs in Theorem \ref{thm:18goninfinite} lie in infinitely many distinct conjugacy classes in the mapping class group ${\rm Mod}_g$, where $g$ is the genus of $X$.

Our proof of Theorem \ref{thm:G18infinite} is obtained via explicit constructions. For example, letting $T = T_{18}$, we will see that
\begin{equation} \label{eq:G18infinite}
\left(S T^4 S T^{-1} S T^{-4} S T \right)^k S T^2 S T^{-2} S T^{-2} \left(S T S T^4 S T^{-1} S T^{-4} \right)^k S T^2
\end{equation}
is a special hyperbolic element of $G_{18}$ for all $k \geq 0$. In the next section, we will provide several other infinite families of special hyperbolic elements. For any $q$, the eigenvalues of a special hyperbolic element of $G_q$ lie in the unit group $\mc{O}_{K_q}^\ast$. When $q = 18$, we have $\mc{O}_{K_{18}}^\ast \cong \Z/2 \times \Z^2$. Despite this, all of the special hyperbolic elements we found in this case had eigenvalues lying in a single cyclic group up to sign. In particular, it turns out that the eigenvalues of the special hyperbolic elements in (\ref{eq:G18infinite}) are all powers of the unit
\begin{equation} \label{eq:K18unit}
u_{18} = 2\lambda_{18}^4 - 4\lambda_{18}^2 + 1 .
\end{equation}
We observed similar coincidences for other values of $q$, especially $q = 7$. In all but one case, the eigenvalues of the special hyperbolic elements in $G_7$ we found were powers of the unit
\begin{equation} \label{eq:K7unit}
u_7 = \lambda_7^2 + \lambda_7
\end{equation}
up to sign, which leads us to suspect that similar constructions to (\ref{eq:G18infinite}) are possible for some other values of $q$. A partial analogue of this phenomenon is known to arise in real quadratic fields via the classical continued fractions associated to $G_3 = \SL(2,\Z)$, see \cite{McM:uniform} and \cite{Wil:fraction} for constructions. However, we emphasize that the associated hyperbolic elements of $G_3$ are not special, indeed $G_3$ does not contain any special hyperbolic elements.

More generally, Theorem \ref{thm:G18infinite}, the occurrence of ``rare'' $G_q$-orbits as shown in Tables \ref{tab:G7orbit} and \ref{tab:G9orbit}, and the behavior shown in Figure \ref{fig:height} (right), seem to suggest a positive answer to the following question.

\begin{ques} \label{ques:Gqinfinite}
For all $q$ such that $K_q$ has degree at least $3$, are there infinitely many distinct $G_q$-orbits contained in $\lambda_q \Q (\lambda_q^2)$?
\end{ques}

Lastly, to emphasize how poorly understood the sets of cusps and hyperbolic fixed points of $G_q$ still are, the following questions are open in all cases where $K_q$ has degree at least $3$.
\begin{enumerate}
    \item Is there an element of $\lambda_q \Q(\lambda_q^2)$ that is not fixed by any element of $G_q$?
    \item Is there an (always terminating) algorithm that takes as input $x \in \lambda_q \Q(\lambda_q^2)$ and outputs whether $x$ is a cusp of $G_q$? A hyperbolic fixed point of $G_q$?
    \item For $q$ odd, does $G_q$ have infinitely many integer cusps?
\end{enumerate}

In \cite{Bou:central} and \cite{HMTY:orbits}, it is conjectured that question (1) has a negative answer for $q = 7,9$. Based on our searches, we expect that question (1) also has a negative answer for $q = 18$, and possibly for $q = 14,30$ as well. Note that for each $q$, a negative answer to (1) implies a positive answer to (2). For question (3), our searches suggest a positive answer for $q = 7,9$. \\

\paragraph{\bf Acknowledgements.} The author thanks Curt McMullen for many inspiring discussions on this topic. During the preparation of this paper, the author was supported by an NSF GRFP under grant DGE-1144152 and by an NSF MSPRF under grant DMS-2303185.

%%%%%%%%%%%%%%%%%%%%%%%%%%%%%%%%%%%%%%%%%%%%%%%%%%
%%%%%%%%%%%%%%%%%%%%%%%%%%%%%%%%%%%%%%%%%%%%%%%%%%
%%%%%%%%%%%%%%%%%%%%%%%%%%%%%%%%%%%%%%%%%%%%%%%%%%

\section{Special hyperbolic elements of $G_{18}$}

We recall material about Hecke triangle groups, their action on $\P^1(\R)$ via M\"{o}bius transformations, and traces of matrix products in $\SL(2,\R)$. We then prove Theorem \ref{thm:G18infinite}. \\

\paragraph{\bf Stabilizers and conjugacy classes in $G_q$.} A matrix $A \in \SL(2,\R)$ is {\em hyperbolic} if $|\Tr(A)| > 2$, {\em parabolic} if $|\Tr(A)| = 2$, and {\em elliptic} if $|\Tr(A)| < 2$. Similarly for elements of $\PSL(2,\R)$. As a subgroup of $\SL(2,\R)$, the Hecke triangle group $G_q$ acts on $\P^1(\R)$ by M\"{o}bius transformations with $\pm I$ acting trivially. We denote by $\ol{G}_q$ the image of $G_q$ in $\PSL(2,\R)$.

The stabilizer in $\ol{G}_q$ of a point in $\P^1(\R)$ is cyclic (Theorem 8.1.2 in \cite{Bea:discrete}). The stabilizer of a hyperbolic fixed point is a maximal cyclic subgroup consisting of hyperbolic elements. {\em Primitive} hyperbolic elements are the generators of these stabilizers. By sending a hyperbolic element to its attracting fixed point, we obtain a bijection between primitive hyperbolic elements of $\ol{G}_q$ and hyperbolic fixed points, and similarly a bijection between conjugacy classes of primitive hyperbolic elements in $\ol{G}_q$ and orbits of hyperbolic fixed points. We record this discussion with the following lemma.

\begin{lem} \label{lem:commonpower}
If $A_1,A_2 \in G_q$ are hyperbolic elements whose attracting fixed points lie in the same $G_q$-orbit, then there is $B \in G_q$ and integers $m_1,m_2 > 0$ such that $B A_1^{m_1} B^{-1} = \pm A_2^{m_2}$.
\end{lem}

The group $\ol{G}_q$ is isomorphic to a free product of two cyclic groups of orders $2$ and $q$, respectively. The images in $\PSL(2,\R)$ of the matrices
\be
S = \left(\begin{matrix}
0 & -1 \\
1 & 0
\end{matrix}\right), \quad
U_q = S T_q = \left(\begin{matrix}
0 & -1 \\
1 & \lambda_q
\end{matrix}\right)
\ee
realize the presentation
\be
\ol{G}_q = \lr{S, U_q \mid S^2 = U_q^q = I} .
\ee
Denote $T = T_q$ and $U = U_q$. Each element in $\ol{G}_q$ can thus be expressed uniquely as a {\em reduced} word in $S$ and $U$ (Theorem 4.1 in \cite{MKS:comb}). These reduced words are powers of $S$ or $U$ or have the form
\be
S^{\eps_1} U^{a_1} S U^{a_2} \cdots S U^{a_{k-1}} S U^{a_k} S^{\eps_k}
\ee
for some $k > 0$, $1 \leq a_1,\dots,a_k \leq q - 1$, and $\eps_1,\eps_k \in \{0,1\}$. A reduced word is {\em cyclically reduced} if it is a power of $S$ or $U$ or if $\eps_1 \neq \eps_k$ above. Up to conjugation by $S$, a cyclically reduced word is a power of $S$ or $U$ or has the form
\be
S U^{a_1} \cdots S U^{a_k}
\ee
for some $k > 0$ and $1 \leq a_1,\dots,a_k \leq q - 1$. Two cyclically reduced words of this form are conjugate in $\ol{G}_q$ if and only if they are cyclic permutations of each other (Theorem 4.2 in \cite{MKS:comb}). For convenience, we will use a weaker version of this fact for words in $S,T$.

\begin{lem} \label{lem:conjST}
Consider elements of $\ol{G}_q$ of the form
\begin{equation} \label{eq:conjST}
S T^{n_1} \cdots S T^{n_k}
\end{equation}
with $k > 0$, $n_j \neq 0$ for $1 \leq j \leq k$ and not all the same sign, and such that in the cyclically ordered sequence $n_1,\dots,n_k$, the maximum number of consecutive $1$'s or consecutive $(-1)$'s is less than $q/2 - 2$. Two such elements are conjugate in $\ol{G}_q$ if and only if they are cyclic permutations of each other.
\end{lem}

\begin{proof}
For each $1 \leq j \leq k$ such that $n_j \geq 2$ or $n_j \leq -2$, substitute $S T^{n_j}$ with $U(SU)^{n_j-1}$ or $S(U^{q-1}S)^{-n_j}$, respectively. For each maximal sequence of consecutive $1$'s in the cyclically ordered sequence $n_1,\dots,n_k$ of length $r$, substitute $(ST)^r$ with $U^r$. For each maximal sequence of consecutive $(-1)$'s of length $r$, substitute $(ST^{-1})^r$ with $S U^{q-r} S$. We have rewritten $S T^{n_1} \cdots S T^{n_k}$ as a concatenation of words $W_1 \cdots W_\ell$ with $1 \leq \ell \leq k$. Each word $W_s$ is a reduced word in $S,U$, and the words coming from positive integers in $n_1,\dots,n_k$ begin and end with $U$, while the words coming from negative integers begin and end in $S$.

If $s,s+1,\dots,s+t$ is a maximal sequence such that $W_s,W_{s+1},\dots,W_{s+t}$ all come from positive integers, then the concatenation $W_s W_{s+1} \cdots W_{s+t}$ has the form
\be
U^{b_0} U(SU)^{a_1-1} U^{b_1} U(SU)^{a_2-1} U^{b_2} \cdots U(SU)^{a_m-1} U^{b_m}
\ee
with $a_1,\dots,a_m \geq 2$ and $0 \leq b_0,\dots,b_m < q/2 - 2$. Combining adjacent $U$'s gives us
\begin{equation} \label{eq:reducedU}
U^{b_0 + 1} (SU)^{a_1-2} S U^{b_1+2} (SU)^{a_2 - 2} S U^{b_2 + 2} \cdots (SU)^{a_m - 2} S U^{b_m + 1}
\end{equation}
which is reduced after removing $(SU)^{a_i-2}$ when $a_i-2 = 0$. If $s,s+1,\dots,s+t$ is a maximal sequence such that $W_s,W_{s+1},\dots,W_{s+t}$ all come from negative integers, then the concatenation $W_s W_{s+1} \cdots W_{s+t}$ has the form
\be
(S U^{q-b_0} S) S(U^{q-1}S)^{a_1} (S U^{q-b_1} S) S(U^{q-1}S)^{a_2} (S U^{q-b_2} S) \cdots S(U^{q-1}S)^{a_m} (S U^{q-b_m} S)
\ee
with $a_1,\dots,a_m \geq 2$, $0 \leq b_0,\dots,b_m < q/2 - 2$. Cancelling adjacent $S$'s and then combining adjacent $U$'s gives us
\begin{equation} \label{eq:reducedS}
S U^{q-b_0-1} S (U^{q-1} S)^{a_1-2} U^{q-b_1-2} S (U^{q-1} S)^{a_2-2} U^{q-b_2-2} S \cdots (U^{q-1} S)^{a_m-2} S U^{q-b_m-1} S
\end{equation}
which is reduced after removing $(U^{q-1} S)^{a_j - 2}$ when $a_j - 2 = 0$. Since the reduced words in (\ref{eq:reducedU}) begin and end in $U$, and the reduced words in (\ref{eq:reducedS}) begin and end in $S$, a concatenation of such words that alternates between (\ref{eq:reducedS}) and (\ref{eq:reducedU}) is also reduced. In this way, we can rewrite $ST^{n_1} \cdots S T^{n_k}$ as a reduced word in $S,U$.

Suppose two elements $S T^{n_1} \cdots S T^{n_k}$ and $S T^{m_1} \cdots S T^{m_\ell}$ of $G_q$ as in (\ref{eq:conjST}) are conjugate in $G_q$. Since $n_1,\dots,n_k$ do not all have the same sign, by conjugating by $S T^{n_1} \cdots S T^{n_j}$ if necessary we may assume $n_1 < 0$ and $n_k > 0$, and we may similarly assume $m_1 < 0$ and $m_\ell > 0$. The associated reduced words in $S,U$ are also conjugate, and by our sign assumptions, they are both cyclically reduced words beginning with $S$ and ending in $U$, and thus are cyclic permutations of each other. Now, the powers $U^p$ that appear in (\ref{eq:reducedU}) all satisfy $1 \leq p < q/2$, and they determine the associated maximal sequence of positive integers in $n_1,\dots,n_k$, with $b_j$ the lengths of maximal sequences of consecutive $1$'s and $a_j$ the integers greater than $1$. Similarly, the powers $U^p$ that appear in (\ref{eq:reducedS}) all satisfy $q/2 < p \leq q - 1$, and they determine the associated maximal sequence of negative integers in $n_1,\dots,n_k$ with $b_j$ the lengths of maximal sequences of consecutive $(-1)$'s and $a_j$ the integers less than $-1$. Thus, since these  cyclically reduced words in $S,U$ are cyclic permutations of each other, the associated sequences $n_1,\dots,n_k$ and $m_1,\dots,m_\ell$ must be cyclic permutations of each other.
\end{proof}

\paragraph{\bf Trace relations.} Recall that for any matrix $A \in \SL(2,\R)$,
\be
A + A^{-1} = \Tr(A) I .
\ee
It follows that for any $A,B \in \SL(2,\R)$,
\begin{equation} \label{eq:trab}
\Tr(AB) + \Tr(AB^{-1}) = \Tr(A)\Tr(B) .
\end{equation}
We will be interested in the traces of certain matrix products of the form
\be
M_{k,\ell} = D^\ell C B^k A
\ee
with $A,B,C,D \in \SL(2,\R)$ and $k,\ell \in \Z$.

\begin{lem} \label{lem:tracerecurrence}
Suppose $\Tr(B) = \Tr(D)$. Then for all integers $k \geq 3$,
\be
\Tr(M_{k,k}) = (\Tr(B)^2 - 1)(\Tr(M_{k-1,k-1}) - \Tr(M_{k-2,k-2})) + \Tr(M_{k-3,k-3}) .
\ee
\end{lem}

\begin{proof}
Applying (\ref{eq:trab}) and cyclic commutivity of trace, we get
\be
\Tr(M_{k,\ell}) = \Tr(B)\Tr(M_{k-1,\ell}) - \Tr(M_{k-2,\ell})
\ee
for all $k \geq 2$, $\ell \geq 0$, and similarly
\be
\Tr(M_{k,\ell}) = \Tr(B)\Tr(M_{k,\ell-1}) - \Tr(M_{k,\ell-2})
\ee
for all $k \geq 0$, $\ell \geq 2$. It follows that for all $k \geq 2$,
\begin{align} \label{eq:trkk}
\Tr(M_{k,k}) &= \Tr(B)^2 \Tr(M_{k-1,k-1}) + \Tr(M_{k-2,k-2}) \nonumber \\ 
&- \Tr(B) (\Tr(M_{k-1,k-2}) + \Tr(M_{k-2,k-1}))
\end{align}
and that
\begin{align} \label{eq:trkk-1}
\Tr(M_{k,k-1}) + \Tr(M_{k-1,k}) &= 2\Tr(B)\Tr(M_{k-1,k-1}) - (\Tr(M_{k-1,k-2}) + \Tr(M_{k-2,k-1})) .
\end{align}

Now suppose $k \geq 3$. By applying (\ref{eq:trkk}) to $\Tr(M_{k,k})$ and $\Tr(M_{k-1,k-1})$, and applying (\ref{eq:trkk-1}) to $\Tr(M_{k-1,k-2}) + \Tr(M_{k-2,k-1})$, we get
\begin{align*}
\Tr(M_{k,k}) &= \Tr(B)^2 \Tr(M_{k-1,k-1}) - \Tr(B) (\Tr(M_{k-1,k-2}) + \Tr(M_{k-2,k-1})) + \Tr(M_{k-2,k-2}) \\
&= \Tr(B)^2 \Tr(M_{k-1,k-1}) + (1 - 2\Tr(B)^2) \Tr(M_{k-2,k-2}) \\
& \quad + \Tr(B)(\Tr(M_{k-2,k-3}) + \Tr(M_{k-3,k-2})) \\
&= (\Tr(B)^2 - 1)\Tr(M_{k-1,k-1}) + (1 - \Tr(B)^2)\Tr(M_{k-2,k-2}) \\
& \quad + \Tr(M_{k-1,k-1}) - \Tr(B)^2 \Tr(M_{k-2,k-2}) + \Tr(B) (\Tr(M_{k-2,k-3}) + \Tr(M_{k-3,k-2})) \\
&= (\Tr(B)^2 - 1)(\Tr(M_{k-1,k-1}) - \Tr(M_{k-2,k-2})) + \Tr(M_{k-3,k-3}) .
\end{align*}
\end{proof}

\begin{lem} \label{lem:trCpowers}
Suppose $\Tr(B) = \Tr(D)$ and that for some integer $n \geq 0$,
\be
\Tr(M_{0,0}) = \Tr(B^n), \quad \Tr(M_{1,1}) = \Tr(B^{n+2}), \quad \Tr(M_{2,2}) = \Tr(B^{n+4}) .
\ee
Then for all integers $k \geq 3$,
\be
\Tr(M_{k,k}) = \Tr(B^{n+2k}) .
\ee
\end{lem}

\begin{proof}
Let $t,t^{-1}$ be the eigenvalues of $B$, so that $\Tr(B^m) = t^m + t^{-m}$ for all $m \in \Z$. We induct on $k$, and we may assume $k \geq 3$. By Lemma \ref{lem:tracerecurrence} and induction on $k$,
\begin{align*}
\Tr(M_{k,k}) &= (t^2 + 1 + t^{-2})(t^{n+2(k-1)} + t^{-n-2(k-1)} - t^{n+2(k-2)} - t^{-n-2(k-2)}) \\
& \quad + (t^{n+2(k-3)} + t^{-n-2(k-3)}) .
\end{align*}
Expanding the right-hand side and canceling terms reduces to
\be
\Tr(M_{k,k}) = t^{n+2k} + t^{-n-2k} = \Tr(B^{n+2k}) .
\ee
\end{proof}

\paragraph{\bf Families of special hyperbolic elements.} A hyperbolic element of $G_q$ is {\em special} if its eigenvalues lie in $K_q$. Denote $\lambda = \lambda_q$. Let $A = \left(\begin{matrix} a & b \\ c & d \end{matrix}\right) \in G_q$ be a special hyperbolic element, and let $t,t^{-1} \in K_q$ be its eigenvalues. Note that $c \neq 0$ since $A$ does not fix the cusp $\infty$. Since $A$ is special, $\Tr(A)^2 - 4 = (t - t^{-1})^2$ is a square in $K_q$. Solving $A \cdot x = (ax + b)/(cx + d) = x$, we see that the fixed points of $A$ are given by $x = ((a - d) \pm (t - t^{-1})) / 2c$. For $q$ odd, since the matrix entries of $G_q$ lie in $\Z[\lambda]$, we have $x \in \Q(\lambda) = \lambda \Q (\lambda^2)$. For $q$ even, by an inductive argument using the generators $S,T_q$, the matrix $A$ has one of the forms
\be
\left(\begin{matrix} a & b_0\lambda \\ c_0\lambda & d \end{matrix}\right), \quad
\left(\begin{matrix} a_0\lambda & b \\ c & d_0\lambda \end{matrix}\right),
\ee
with $a_0,b_0,c_0,d_0 \in \Z[\lambda^2]$ (see Corollary 1 in \cite{Ros:fraction}). Since $\Tr(A) \in K_q = \Q(\lambda^2)$ is nonzero, $A$ is of the first form and $x = ((a-d) \pm (t-t^{-1}))/2\lambda c_0 \in \lambda \Q(\lambda^2)$. Thus, fixed points of special hyperbolic elements of $G_q$ lie in $\lambda \Q(\lambda^2)$.

We now apply Lemma \ref{lem:trCpowers} to produce infinite families of orbits of special hyperbolic fixed points. Denote $T = T_{18}$ and $\lambda = \lambda_{18}$. For $n_1,\dots,n_k$ a finite sequence of nonzero integers, define
\be
M(n_1,\dots,n_k) = S T^{n_k} \cdots S T^{n_1} \in G_{18} .
\ee
Additionally, denote by $(n_1,\dots,n_k)^m$ the concatenation of $m$ copies of $n_1,\dots,n_k$. Lastly, recall the unit $u_{18} = 2\lambda^4 - 4\lambda^2 + 1$ from (\ref{eq:K18unit}) and denote $u = u_{18}$. All of the computational checks in the proof of the following theorem were carried out with SageMath.

\begin{thm} \label{thm:G18families}
For all integers $k \geq 0$, the following elements of $G_{18}$ are special hyperbolic elements. The fixed points of each family form infinitely many distinct $G_{18}$-orbits in $\lambda\Q(\lambda^2)$.
\begin{align*}
& M(2, (-4, -1, 4, 1)^k, -2, -2, 2, (1, -4, -1, 4)^k) \\
& M(4, (2, -2, -2, 2)^k, 1, -4, -1, (2, 2, -2, -2)^k) \\
& M(-4, (-1, 8, 1, -2)^k, -2, 1, 2, (-2, -1, 8, 1)^k) \\
& M(-1, (-4, 2, 1, -2)^k, -2, 1, 8, (1, -1, -1, 16)^k) \\
& M(16, (1, -2, -1, 8)^k, -1, -1, 1, (8, 1, -2, -1)^k) \\
& M(4, (2, -2, -2, 2)^k, 2, -2, -1, 4, 1, -2, -2, (2, 2, -2, -2)^k) \\
& M(4, (1, -2, -4, 2)^k, 1, -2, -2, 4, 2, -2, -1, (2, 4, -2, -1)^k) \\
& M(2, (-4, -1, 4, 1)^k, -4, -1, 2, 2, -2, -1, 4, (1, -4, -1, 4)^k) \\
& M(2, (-2, -1, 8, 1)^k, -2, -1, 4, 2, -4, -1, 2, (1, -8, -1, 2)^k) \\
& M(-4, (-1, 8, 1, -2)^k, -1, 8, -1, -1, 1, 8, 1, (-2, -1, 8, 1)^k) \\
& M(2, (-8, -1, 2, 1)^k, -8, -1, 1, 2, -1, -1, 8, (1, -2, -1, 8)^k) \\
& M(2, (-1, -1, 16, 1)^k, -1, -1, 8, 2, -8, -1, 1, (1, -16, -1, 1)^k) \\
& M(16, (1, -2, -1, 8)^k, 1, -2, -2, 1, 2, -2, -1, (8, 1, -2, -1)^k)
\end{align*}
\end{thm}

\begin{proof}
The proofs for each family are similar, so we only present the proof for the first family
\be
M(2, (-4, -1, 4, 1)^k, -2, -2, 2, (1, -4, -1, 4)^k)
\ee
in detail. Up to cyclic permutation, the repeating parts in this family are equal to
\be
M(4,1,-4,-1) =
\left(\begin{matrix}
    4\lambda^2 + 1 & 16\lambda^3 \\
    4\lambda^3 & 16\lambda^4 - 4\lambda^2 + 1
\end{matrix}\right)
\ee
and the non-repeating part (the $k = 0$ case) is equal to
\be
M(2,2,-2,-2) =
\left(\begin{matrix}
    4\lambda^2 + 1 & 8\lambda^3 \\
    8\lambda^3 & 16\lambda^4 - 4\lambda^2 + 1
\end{matrix}\right)
\ee
and both of these matrices have trace
\be
16\lambda^4 + 2 = u^2 + u^{-2} .
\ee
This family has the form
\be
M_k = D^k C B^k A, \quad k \geq 0,
\ee
with $A = M(2)$, $B = M(-4,-1,4,1)$, $C = M(-2,-2,2)$, $D = M(1,-4,-1,4)$. By the above calculations,
\be
\Tr(M_0) = \Tr(B) = \Tr(D) = u^2 + u^{-2} .
\ee
Additionally, we check using Sage that
\be
\Tr(M_1) = 1847952\lambda^4 - 3838464\lambda^2 + 1391618 = u^6 + u^{-6}
\ee
\be
\Tr(M_2) = 110983509904\lambda^4 - 235185378816\lambda^2 + 85747037186 = u^{10} + u^{-10}
\ee
and thus by Lemma \ref{lem:trCpowers},
\be
\Tr(M_k) = u^{4k+2} + u^{-4k-2}
\ee
for all $k \geq 0$. Since $u \in K_q$ and $|u| \approx 15.582 > 1$, this means $M_k$ is a special hyperbolic element of $G_{18}$ for all $k \geq 0$.

Suppose that $k_1,k_2 \geq 0$ are integers such that the attracting fixed points of $M_{k_1},M_{k_2}$ lie in the same $G_{18}$-orbit. By Lemma \ref{lem:commonpower}, there are integers $m_1,m_2 > 0$ such that $M_{k_1}^{m_1}$ and $\pm M_{k_2}^{m_2}$ are conjugate in $G_{18}$. The sequences
\be
(2, (-4,-1,4,1)^{k_j}, -2, -2, 2, (1,-4,-1,4)^{k_j})^{m_j}
\ee
defining $M_{k_j}^{m_j}$ do not contain multiple consecutive $1$'s or $(-1)$'s, so Lemma \ref{lem:conjST} tells us that these sequences are cyclic permutations of each other. By counting the number of $2$'s in each sequence, we see that $2 m_1 = 2 m_2$ and thus $m_1 = m_2$. Then by counting the number of $4$'s, we see that $2 k_1 m_1 = 2 k_2 m_2 = 2 k_2 m_1$ and thus $k_1 = k_2$. Thus, the attracting fixed points of $M_k$, $k \geq 0$, lie in infinitely many distinct $G_{18}$-orbits.
\end{proof}

%%%%%%%%%%%%%%%%%%%%%%%%%%%%%%%%%%%%%%%%%%%%%%%%%%
%%%%%%%%%%%%%%%%%%%%%%%%%%%%%%%%%%%%%%%%%%%%%%%%%%
%%%%%%%%%%%%%%%%%%%%%%%%%%%%%%%%%%%%%%%%%%%%%%%%%%

\section{$\lambda_q$-continued fraction expansions}

Many questions about the Hecke triangle groups $G_q$ can be studied computationally using a family of continued fraction algorithms introduced by Rosen \cite{Ros:fraction}. Throughout, we denote $\lambda = \lambda_q$, $K = K_q$, and $T = T_q$. Rosen showed that any $x \in \R$ can be expressed as a {\em $\lambda$-continued fraction}
\begin{equation} \label{eq:lambdaCF}
x = [a_0,a_1,a_2,\dots] = a_0 \lambda - \frac{1}{a_1 \lambda - \frac{1}{a_2 \lambda - \cdots}}
\end{equation}
\noindent
with $a_0 \in \Z$, $a_i \in \Z \sm \{0\}$ for $i \geq 1$, in a unique way by requiring
\begin{align*}
x - a_0\lambda &\in (-\lambda/2, \lambda/2] \\
-1/(x - a_0\lambda) - a_1\lambda &\in (-\lambda/2, \lambda/2] \\
& \vdots
\end{align*}
The cusps $G_q \cdot \infty$ are precisely the elements of $\P^1(\R)$ that can be expressed as finite $\lambda$-continued fractions $[a_0,a_1,\dots,a_N]$. If there exists $n \geq 1$ and $i_0 \geq 0$ such that $a_i = a_{i+n}$ for all $i \geq i_0$, we say that the $\lambda$-continued fraction is {\em preperiodic}, and {\em periodic} if we can take $i_0 = 0$. If $x$ is a periodic $\lambda$-continued fraction, we denote this by
\be
x = \ol{[a_0,a_1,\dots,a_n]}
\ee
where $n+1$ is the minimal period, and $x$ is fixed by
\be
M_x = ST^{-a_n} \cdots ST^{-a_0} \in G_q .
\ee

For all $x \in \P^1(\R)$, we have
\be
S \cdot (-x) = -(S \cdot x), \quad T \cdot (-x) = -(T^{-1} \cdot x) .
\ee
It follows that the set of cusps and the set of hyperbolic fixed points are preserved under negation. Since $\lambda \Q(\lambda^2)$ is invariant under negation, the set of special hyperbolic fixed points is also preserved under negation. Since $\lambda/2$ is a cusp \cite{Ros:fraction}, for any periodic $\lambda$-continued fraction, negation simply negates its periodic part. Suppose $x = \ol{[a_0,a_1,\dots,a_n]}$ is fixed by a hyperbolic element $M_x = S T^{-a_n} \cdots S T^{-a_0}$, and let $x^\pr = \ol{[a_n,a_{n-1},\dots,a_0]}$. Since $M_x$ and $M_x^{-1}$ have the same fixed points and $S M_x^{-1} S^{-1} = S T^{a_0} \cdots S T^{a_n}$, we see that $M_x$ also fixes $S \cdot (-x^\pr) = 1/x^\pr$. The periodic $\lambda$-continued fractions obtained by negating and reversing the periodic part do not necessarily lie in the same $G_q$-orbit. In this way, we may obtain $1$, $2$, or $4$ distinct $G_q$-orbits.

Below, we summarize the computer searches we carried out to find new $G_q$-orbits of special hyperbolic fixed points. All of our computations were done using SageMath, which supports arithmetic in number fields and arbitrary precision numerical computations. \\

\paragraph{\bf $G_7$-orbits.} Let $q = 7$. Then $\lambda$ has one other Galois conjugate with absolute value greater than $1$, given by $\lambda^\pr = -\lambda^2 + 2$. The ring of integers of $K_q$ is $\Z[\lambda]$, and the positive unit group is generated by $\lambda, -\lambda^\pr$. We counted the number of elements of each $G_7$-orbit in $\Z \cap [1,10^6]$ by computing the associated $\lambda$-continued fraction expansions, and similarly for a large subset of the positive unit group and the ring of integers. The results are reported in Table \ref{tab:G7orbit}. All of the $\lambda$-continued fraction expansions we computed were either finite or preperiodic.

We found several new $G_7$-orbits of special hyperbolic fixed points, and Table \ref{tab:Gqperiod} lists the associated periodic $\lambda$-continued fractions, up to negation and reversal. We obtain a lower bound on the number $N_7$ of $G_7$-orbits in $\P^1(K_q)$ of
\be
N_7 \geq 1 + 1 + 4 + 4 + 4 + 4 + 2 + 4 = 24
\ee
Note that only $13$ of these $24$ orbits appear in the Table \ref{tab:G7orbit}. In particular, even up to negation and reversal of periodic parts, the orbit of
\be
\frac{531}{7} \lambda^2 + \frac{402}{7} \lambda - \frac{319}{7} = \ol{[169,2,1,-1,-1,-2]}
\ee
does not appear. We found this orbit of special hyperbolic fixed points by checking all periodic $\lambda$-continued fractions of the form $\ol{[a_1,a_2,\dots,a_6]}$ with $\prod_{j=1}^6 a_i = \pm 2^2 \cdot 13^2$. Lastly, for all but the last row in the $q = 7$ section of Table \ref{tab:Gqperiod}, the eigenvalues of the associated special hyperbolic elements are powers of the unit
\be
u_7 = \lambda_7^2 + \lambda_7 .
\ee
We were unable to prove a version of Theorem \ref{thm:G18families} in the $q = 7$ case though.

\begin{center}
\begin{table}
\caption{$q = 7$: Counts of elements of $\Z$, $\mc{O}_{K_q}^\ast$, and $\mc{O}_{K_q}$ by $G_q$-orbit}
\label{tab:G7orbit}
\renewcommand{\arraystretch}{1.5}
\begin{tabular}{|c|c|c|c|}
    \hline
    & & & \\ [-4ex]
     Orbit representative & $\Z \cap [1,10^6]$ & \shortstack{$\lambda^a (-\lambda^\pr)^b$ \\ $a,b \in \Z \cap [-10^2,10^2]$} & \shortstack{$a + b\lambda^2$ \\ $a,b \in \Z \cap [-10^3,10^3]$} \\
     \hline
     $\infty$ & 999973 & 28857 & 3003907 \\
     \hline
     % [1, -1]
     $\lambda^2 - 1$ & 0 & 11446 & 991616 \\
     \hline
     % [13, 1, 2, -13, -2, -1]
     $\frac{37}{7} \lambda^2 + \frac{29}{7} \lambda - \frac{13}{7}$ & 0 & 39 & 3542 \\
     \hline
     % [-13, -1, -2, 13, 2, 1]
     $-\frac{37}{7} \lambda^2 - \frac{29}{7} \lambda + \frac{13}{7}$ & 0 & 39 & 3542 \\
     \hline
     % [13, -2, -1, -13, 1, 2]
     $\frac{43}{7}\lambda^2 + \frac{32}{7}\lambda - \frac{31}{7}$ & 0 & 10 & 645 \\
     \hline
     % [-13, 2, 1, 13, -1, -2]
     $-\frac{43}{7}\lambda^2 - \frac{32}{7}\lambda + \frac{31}{7}$ & 0 & 10 & 645 \\
     \hline
     % [58, 1, -1, 1, 2, 2, 1, -1, 1, -58, -1, 1, -1, -2, -2, -1, 1, -1]
     $26\lambda^2 + \frac{39}{2}\lambda - \frac{31}{2}$ & 15 & 0 & 38 \\
     \hline
     % [258, -1, 1, 2, -1, -1, 1, -1, 1, 3, 1, 1, -1, -1, -6, 7, -2, -1, -1, -2, -1, -7, 1]
     $\frac{782}{7}\lambda^2 + \frac{636}{7}\lambda - \frac{428}{7}$ & 8 & 0 & 28 \\
     \hline
     % [-258, 1, -1, -2, 1, 1, -1, 1, -1, -3, -1, -1, 1, 1, 6, -7, 2, 1, 1, 2, 1, 7, -1]
     $-\frac{782}{7}\lambda^2 - \frac{636}{7}\lambda + \frac{428}{7}$ & 4 & 0 & 28 \\
     \hline
     % [46, 1, -1, 2, 1, 4, -2, -1]
     $\frac{825}{43} \lambda^2 + \frac{689}{43} \lambda - \frac{375}{43}$ & 0 & 0 & 3 \\
     \hline
     % [-46, -1, 1, -2, -1, -4, 2, 1]
     $-\frac{825}{43} \lambda^2 - \frac{689}{43} \lambda + \frac{375}{43}$ & 0 & 0 & 3 \\
     \hline
     % [169, 1, 2, -1, -2, -1]
     $\frac{529}{7} \lambda^2 + \frac{401}{7} \lambda - \frac{313}{7}$ & 0 & 0 & 2 \\
     \hline
     % [-169, -1, -2, 1, 2, 1]
     $-\frac{529}{7} \lambda^2 - \frac{401}{7} \lambda + \frac{313}{7}$ & 0 & 0 & 2 \\
     \hline
\end{tabular}
\end{table}
\end{center}

\subsection*{$G_9$-orbits.} Let $q = 9$. Again, the ring of integers of $K_q$ is $\Z[\lambda]$, and the positive unit group is generated by $\lambda$ and $-\lambda^\pr = \lambda^2 - 2$. Table \ref{tab:G9orbit} counts the number of elements of each $G_9$-orbit in a large subset of the rational integers, the positive units, and the ring of integers, and Table \ref{tab:Gqperiod} lists the associated periodic $\lambda$-continued fractions. We get a lower bound on the number $N_9$ of $G_9$-orbits in $\P^1(K_9)$ of
\be
N_9 \geq 1 + 4 + 4 + 4 = 13 .
\ee
All of the $\lambda$-continued fraction expansions we computed were either finite or preperiodic.

\begin{center}
\begin{table}
\vspace{10pt}
\caption{$q = 9$: Counts of elements of $\Z$, $\mc{O}_{K_q}^\ast$, and $\mc{O}_{K_q}$ by $G_q$-orbit}
\label{tab:G9orbit}
\renewcommand{\arraystretch}{1.5}
\begin{tabular}{|c|c|c|c|c|c|}
    \hline
    & & & \\ [-4ex]
     Orbit representative & $\Z \cap [1,10^6]$ & \shortstack{$\lambda^a (-\lambda^\pr)^b$ \\ $a,b \in \Z \cap [-10^2,10^2]$} & \shortstack{$a + b\lambda^2$ \\ $a,b \in \Z \cap [-10^3,10^3]$} \\
     \hline
     Cusps & 676292 & 27225 & 2707373 \\
     \hline
     $2\lambda + 2$ & 152442 & 6247 & 611123 \\
     \hline
     $-2\lambda - 2$ & 152957 & 6247 & 611123 \\
     \hline
     $8\lambda + 8$ & 8816 & 324 & 35528 \\
     \hline
     $-8\lambda - 8$ & 8650 & 324 & 35528 \\
     \hline
     $\frac{5}{2}\lambda^22 + \frac{9}{2}\lambda + \frac{3}{2}$ & 399 & 16 & 1604 \\
     \hline
     $-\frac{5}{2}\lambda^22 - \frac{9}{2}\lambda - \frac{3}{2}$ & 412 & 16 & 1604 \\
     \hline
     $\frac{80}{57}\lambda^2 - \frac{74}{57}\lambda - \frac{86}{19}$ & 19 & 1 & 59 \\
     \hline
     $-\frac{80}{57}\lambda^2 + \frac{74}{57}\lambda + \frac{86}{19}$ & 13 & 1 & 59 \\
     \hline
\end{tabular}
\end{table}
\end{center}

\paragraph{\bf $G_{18}$-orbits.} Let $q = 18$. We carried out similar searches as in the $G_7$ and $G_9$ cases, for elements of the form
\begin{align*}
\lambda n, & \quad n \in \Z \cap [1,10^6] \\
\lambda \lambda_1^a \lambda_2^b, & \quad a,b \in \Z \cap [-10^2,10^2] \\
\lambda^3 a + \lambda^5 b, & \quad a,b \in \Z \cap [-10^3,10^3]
\end{align*}
where $\lambda_1 = \lambda^2 - 1$ and $\lambda_2 = \lambda^2 - 2$. In this case, we found a much larger number of orbits, so we only report the associated periodic $\lambda$-continued fractions up to negation and reversal in Table \ref{tab:G18period}. All of the $\lambda$-continued fraction expansions we computed were either finite or preperiodic. Moreover, all of the periodic $\lambda$-continued fractions $\ol{[a_1,a_2,\dots,a_n]}$ we found satisfied the following properties:
\begin{itemize}
    \item $n$ is divisible by $4$
    \item for $1 \leq j \leq n$, there is $k \geq 0$ such that $a_j = \pm 2^k$
    \item the number of positive and negative $a_j$'s is equal
    \item $\prod_{j=1}^n a_j = 2^n$
\end{itemize}
Note that only finitely many periodic parts of each length satisfy the above properties. We systematically checked all periodic parts of length $4$, $8$, and $12$ satisfying the above properties, and listed the periodic parts up to negation and reversal that yield special hyperbolic fixed points in Table \ref{tab:G18period}. \\

\begin{center}
\begin{table}
\caption{$q = 18$, periodic $\lambda$-continued fractions in $\lambda \Q(\lambda^2)$}
\label{tab:G18period}
\begin{tabular}{| c | c |}
 \hline \\ [-2ex]
 $\ol{[2, 2, -2, -2]}$ & $\ol{[4, 4, -4, 1, 2, -1, -4, -1, 4, 1, -2, -1]}$ \\ [1ex]
 $\ol{[4, 1, -4, -1]}$ & $\ol{[4, 4, -4, -1, 2, 1, -4, -1, 4, -1, -2, 1]}$ \\ [1ex]
 $\ol{[4, 2, -1, -2]}$ & $\ol{[4, 4, -4, -2, -1, 2, 1, -4, -1, 2, 1, -2]}$ \\ [1ex]
 $\ol{[8, 1, -2, -1]}$ & $\ol{[8, 2, -2, -1, 2, 2, -2, -2, -2, 1, 2, -2]}$ \\ [1ex]
 $\ol{[16, 1, -1, -1]}$ & $\ol{[8, 2, 2, -1, -2, 2, -2, -2, 2, 1, -2, -2]}$ \\ [1ex]
 $\ol{[4, 2, -2, -1, 4, 1, -2, -2]}$ & $\ol{[8, 2, 1, -2, -1, 2, -8, -2, -1, 2, 1, -2]}$ \\ [1ex]
 $\ol{[4, 2, -4, -1, 2, 2, -2, -1]}$ & $\ol{[8, 2, -8, -1, -2, 1, 2, -2, -2, 1, 2, -1]}$ \\ [1ex]
 $\ol{[8, 1, -4, -1, 8, -1, -1, 1]}$ & $\ol{[8, 2, -4, 1, 1, -2, -2, 2, 1, -1, -4, -2]}$ \\ [1ex]
 $\ol{[8, 2, -8, -1, 1, 2, -1, -1]}$ & $\ol{[8, 1, -4, -1, 8, 1, -2, -2, 1, 2, -2, -1]}$ \\ [1ex]
 $\ol{[16, 1, -2, -2, 1, 2, -2, -1]}$ & $\ol{[16, 2, -2, -1, 2, 2, -1, -2, -2, 1, 2, -2]}$ \\ [1ex]
 $\ol{[4, 2, 2, -1, -2, 2, -4, -2, 2, 1, -2, -2]}$ & $\ol{[16, 2, 2, -1, -2, 2, -1, -2, 2, 1, -2, -2]}$ \\ [1ex]
 $\ol{[4, 2, -2, -2, 2, 1, -4, -1, 2, 2, -2, -2]}$ & $\ol{[16, 1, 2, -1, -2, 1, -16, -1, -2, 1, 2, -1]}$ \\ [1ex]
 $\ol{[4, 2, -2, -1, 2, 2, -4, -2, -2, 1, 2, -2]}$ & $\ol{[16, 1, -2, -1, 8, -1, -1, 1, 8, 1, -2, -1]}$ \\ [1ex]
 $\ol{[4, 2, -4, -1, 2, 1, -4, -2, 4, -1, -2, 1]}$ & $\ol{[16, 1, -4, -2, 1, 2, -2, -1, 8, -1, -1, 1]}$ \\ [1ex]
 $\ol{[4, 2, -4, 1, 2, -1, -4, -2, 4, 1, -2, -1]}$ & $\ol{[16, -1, -4, 2, 1, -2, -2, 1, 8, 1, -1, -1]}$ \\ [1ex]
 $\ol{[4, 2, -4, -1, 4, 1, -2, -2, 2, 1, -4, -1]}$ & $\ol{[32, 1, 1, -1, -4, 1, -8, -1, -4, 1, 1, -1]}$ \\ [1ex]
 \hline
\end{tabular}
\end{table}
\end{center}

\paragraph{\bf $G_q$-orbits for other $q$.} For other values of $q$, we did some ad hoc searches to find additional special hyperbolic fixed points, which are reported in Table \ref{tab:Gqperiod} in terms of periodic $\lambda$-continued fractions. For space reasons, we sometimes use exponential notation $n^k$ to denote $k$ consecutive occurrences of $n$. Many of the orbits in Table \ref{tab:Gqperiod} have previously appeared in the literature \cite{AS:special}, \cite{Bou:central}, \cite{HMTY:orbits}, and two of the examples with $q = 16,30$ were recently independently discovered in \cite{ABHH:special}.

\begin{center}
\begin{table}
\caption{Periodic $\lambda$-continued fractions in $\lambda\Q(\lambda^2)$}
\label{tab:Gqperiod}
\begin{tabular}{| c | c |}
\hline \\ [-2ex]
$q = 7$ & \\
 & \\ [-2ex]
 $\lambda^2 - 1$ & $\ol{[1,-1]}$ \\ [1ex]
 $\frac{37}{7} \lambda^2 + \frac{29}{7} \lambda - \frac{13}{7}$ & $\ol{[13, 1, 2, -13, -2, -1]}$ \\ [1ex]
 $\frac{529}{7} \lambda^2 + \frac{401}{7} \lambda - \frac{313}{7}$ & $\ol{[169, 1, 2, -1, -2, -1]}$ \\ [1ex]
 $\frac{531}{7} \lambda^2 + \frac{402}{7} \lambda - \frac{319}{7}$ & $\ol{[169, 2, 1, -1, -1, -2]}$ \\ [1ex]
 $\frac{825}{43} \lambda^2 + \frac{689}{43} \lambda - \frac{375}{43}$ & $\ol{[46, 1, -1, 2, 1, 4, -2, -1]}$ \\ [1ex]
 $26 \lambda^2 + \frac{39}{2} \lambda - \frac{31}{2}$ & $\ol{[58, 1, -1, 1, 2, 2, 1, -1, 1, -58, -1, 1, -1, -2, -2, -1, 1, -1]}$ \\ [1ex]
 $\frac{782}{7} \lambda^2 + \frac{636}{7} \lambda - \frac{428}{7}$ & $\ol{[258, -1, 1, 2, -1, -1, 1, -1, 1, 3, 1, 1, -1,}$ \\ & $\ol{-1, -6, 7, -2, -1, -1, -2, -1, -7, 1]}$ \\ [1ex]
 \hline
 \hline \\ [-2ex]
 $q = 9$ & \\
 & \\ [-2ex]
 $2 \lambda + 2$ & $\ol{[3, -4, 1, 1]}$ \\ [1ex]
 $8 \lambda + 8$ & $\ol{[12, -1, 3, 1, -2, -18, -1, 40, 3, 1^3]}$ \\ [1ex]
 $\frac{5}{2} \lambda^2 + \frac{9}{2} \lambda + \frac{3}{2}$ & $\ol{[10, 83, -2, (-1)^3, 2, 1, 4, -1, -1, -4, 1, -1]}$ \\ [1ex]
 \hline
 \hline \\ [-2ex]
 $q = 14$ & \\
 & \\ [-2ex]
 $\lambda^3 - 3 \lambda$ & $\ol{[1, 1, -1, -1]}$ \\ [1ex]
 $2 \lambda^5 - 6 \lambda^3 + 3 \lambda$ & $\ol{[9, -3, -1, -2, -1, -2, -1, -2, -9, 3, 1, 2, 1, 2, 1, 2]}$ \\ [1ex]
 $10 \lambda^5 - 32 \lambda^3 + 18 \lambda$ & $\ol{[41, 2, -1, -3, (-1)^4, 2, 1^4, -1]}$ \\ [1ex]
 \hline
 \hline \\ [-2ex]
 $q = 16$ & \\
 & \\ [-2ex]
 $\lambda^3 - 3 \lambda$ & $\ol{[1, 2, 1, 2, -1, -2, -1, -2]}$ \\ [1ex]
 \hline
 \hline \\ [-2ex]
 $q = 20$ & \\
 & \\ [-2ex]
 $\lambda^5 - 4 \lambda^3 + 2 \lambda$ & $\ol{[2, 1^3, -2, (-1)^3]}$ \\ [1ex]
 $2\lambda^3 - 6\lambda$ & $\ol{[2, 1, -1, -2, (-1)^9, 4, 13, 1, 1, -1, -1, -2, -1, 3, -1,}$ \\ & $\ol{5, 1, 1, 2, 1, -6, (-1)^2, 4, (-1)^6, 1, 2, 1^5, -2, 1^8]}$ \\ [1ex]
 \hline
 \hline \\ [-2ex]
 $q = 24$ & \\
 & \\ [-2ex]
 $\lambda^5 - 3 \lambda^3 + \lambda$ & $\ol{[5, 1, 1, -6, (-1)^4, 1, 1, -5, -1, -1, 6, 1^4, -1, -1]}$ \\ [1ex]
 $\frac{1}{2}\lambda^7 - 3\lambda^5 + \frac{11}{2}\lambda^3 - \frac{5}{2}\lambda$ & $\ol{[3, -2, -1, 1^{11}, -3, (-1)^{10}, 1]}$ \\ [1ex]
 \hline
 \hline \\ [-2ex]
 $q = 30$ & \\
 & \\ [-2ex]
 $\lambda^7 - 6\lambda^5 + 10\lambda^3 - 4\lambda$ & $\ol{[4, 1^5, -4, (-1)^5]}$ \\ [1ex]
 $\frac{52}{29}\lambda^7 - \frac{207}{29}\lambda^5 + \frac{243}{29}\lambda^3 - \frac{90}{29}\lambda$ & $\ol{[29, -1, -1, -2, (-1)^3, 1, -1, -2, -1, 3, 1^3, 4, 2, -1, 1^{10},}$ \\ & $\ol{18, 1, -5, 2, -2, (-1)^3, 1^4, -1, 1, 3, 1^5, 4, 1, 1, -19, (-1)^8,}$ \\ & $\ol{-2, 1^5, 5, -29, 1, 1, 2, 1^3, -1, 1, 2, 1, -3, (-1)^3, -4, -2, 1,}$ \\ & $\ol{(-1)^{10}, -18, -1, 5, -2, 2, 1^3, (-1)^4, 1, -1, -3, (-1)^5,}$ \\ & $\ol{-4, -1, -1, 19, 1^8, 2, (-1)^5, -5]}$ \\ [1ex]
 \hline
\end{tabular}
\end{table}
\end{center}

For $q = 7,9,18$, our computations provide strong evidence that every element of $\lambda\Q(\lambda^2)$ is fixed by a nontrivial element of $G_q$. This was previously conjectured in the cases $q = 7,9$ in \cite{Bou:central} and \cite{HMTY:orbits}. For other values of $q$, it seems more difficult to investigate this question empirically.

To illustrate some of this evidence, suppose for simplicity that $K_q$ has class number $1$. Then every element of $\P^1(K_q)$ can be expressed as $[a : b]$ with $a,b \in \mc{O}_{K_q}$ relatively prime. Define a {\em height} on $\P^1(K_q)$ by
\begin{equation*}
h([a : b]) = \prod_{\sigma : K_q \hookrightarrow \R} \left(|\sigma(a)| + |\sigma(b)|\right)
\end{equation*}
where the product is over the real embeddings of $K_q$. Note that $K_q$ is totally real. The function $h$ is well-defined since any other expression for $[a : b] \in \P^1(K_q)$ differs by simultaneous multiplication of $a,b$ by a unit $u \in \mc{O}_{K_q}^\ast$, which satisfies $\prod_{\sigma} |\sigma(u)| = 1$. For elements $x \in \lambda\Q(\lambda^2)$, we will consider the height of $x/\lambda$. We refer to \cite{McM:heights} for more discussion on heights in the context of triangle groups and Veech groups.

Figure \ref{fig:height} illustrates the behavior of $\log(h)$ of a ``typical'' element of $\lambda \Q(\lambda^2)$ under the $\lambda$-continued fraction algorithm. The cases $q = 7,9,18$ all behave as in the left image in Figure \ref{fig:height}. The cases $q = 14,30$ behave as in Figure \ref{fig:height2}. The remaining cases (where $K_q$ has degree at least $3$) all seem to behave as in the right image in Figure \ref{fig:height}.

\begin{figure}[t]
\centering
\includegraphics[width=0.49\textwidth]{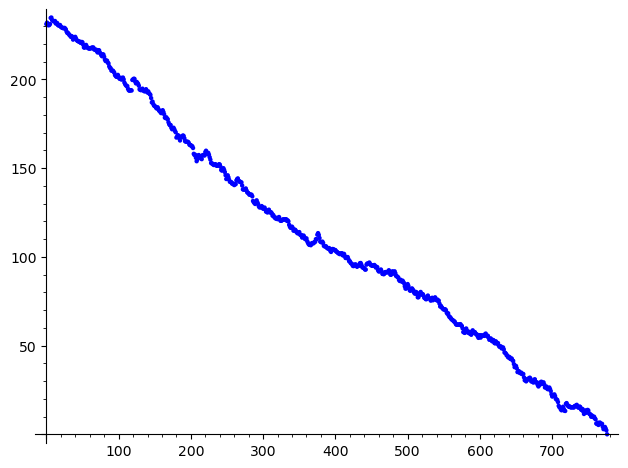} \includegraphics[width=0.49\textwidth]{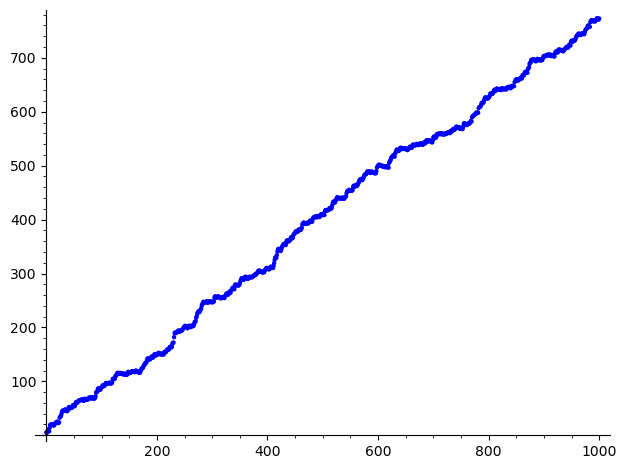}
\caption{Left: $q = 7$, behavior of $\log h(x)$ under the $\lambda$-continued fraction algorithm for a randomly chosen cusp $x \in \Z$ with $x \approx 10^{50}$. Similar behavior observed for $q = 9, 18$. Right: $q = 11$, behavior of $\log h(x)$ under the $\lambda$-continued fraction algorithm starting at $x = 2$. Similar behavior expected for all odd $q \geq 11$, and all even $q \geq 16$ except for $q = 18, 30$.}
\label{fig:height}
\end{figure}

\begin{figure}[t]
\centering
\includegraphics[width=0.5\textwidth]{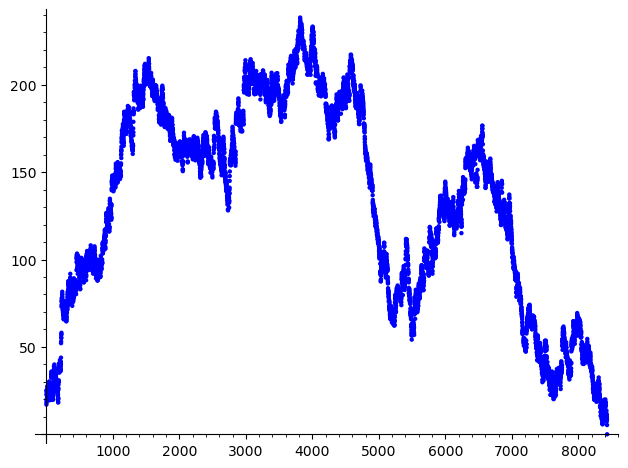}
\caption{$q = 14$, behavior of $\log h(x)$ under the $\lambda$-continued fraction algorithm starting at the cusp $x = 24\lambda^3$. Similar behavior observed for $q = 30$.}
\label{fig:height2}
\end{figure}

%%%%%%%%%%%%%%%%%%%%%%%%%%%%%%%%%%%%%%%%%%%%%%%%%%
%%%%%%%%%%%%%%%%%%%%%%%%%%%%%%%%%%%%%%%%%%%%%%%%%%
%%%%%%%%%%%%%%%%%%%%%%%%%%%%%%%%%%%%%%%%%%%%%%%%%%

\bibliographystyle{math}
\bibliography{my.bib}

{\small
\noindent
Email: kgwinsor@gmail.com

\noindent
Department of Mathematics, Stony Brook University, Stony Brook, New York, USA
}

\end{document}